\documentclass[11pt]{amsart}
\usepackage{eucal,fullpage,times,amsmath,amsthm,amssymb,mathrsfs,stmaryrd,color}
\usepackage[all]{xy}
\usepackage{url}

\newcommand{\calO}{{\mathcal{O}}}

\newcommand{\calM}{{\mathcal{M}}}

\newcommand{\calQ}{\mathcal{Q}}
\newcommand{\calE}{\mathcal{E}}

\newcommand{\calC}{\mathcal{C}}

\newcommand{\A}{\mathbf{A}}
\newcommand{\G}{\mathbf{G}}
\newcommand{\Z}{\mathbf{Z}}

\newcommand{\F}{\mathbf{F}}

\renewcommand{\P}{\mathbf{P}}

\newcommand{\Spec}{{\mathrm{Spec}}}

\newcommand{\Hom}{\mathrm{Hom}}

\newcommand{\Coh}{\mathrm{Coh}}

\newcommand{\Vect}{\mathrm{Vect}}

\newcommand{\D}{\mathrm{D}}
\newcommand{\R}{\mathrm{R}}
\renewcommand{\L}{\mathrm{L}}
\newcommand{\W}{\mathrm{W}}
\newcommand{\Pic}{\mathrm{Pic}}
\newcommand{\et}{\mathrm{\acute{e}t}}

\newcommand{\pr}{\mathrm{pr}}

\newcommand{\Bl}{\mathrm{Bl}}

\newcommand{\gr}{\mathrm{gr}}

\newcommand{\sh}{\mathrm{sh}}

\newcommand{\comment}[1]{}

\DeclareMathOperator{\colim}{colim}

\begin{document}

\bibliographystyle{alpha}

\newtheorem{theorem}{Theorem}[section]
\newtheorem*{theorem*}{Theorem}
\newtheorem*{condition*}{Condition}
\newtheorem*{definition*}{Definition}
\newtheorem{proposition}[theorem]{Proposition}
\newtheorem{lemma}[theorem]{Lemma}
\newtheorem{corollary}[theorem]{Corollary}
\newtheorem{claim}[theorem]{Claim}
\newtheorem{claimex}{Claim}[theorem]

\theoremstyle{definition}
\newtheorem{definition}[theorem]{Definition}
\newtheorem{question}[theorem]{Question}
\newtheorem{remark}[theorem]{Remark}
\newtheorem{example}[theorem]{Example}
\newtheorem{condition}[theorem]{Condition}
\newtheorem{warning}[theorem]{Warning}
\newtheorem{notation}[theorem]{Notation}

\title{$p$-divisibility for coherent cohomology}
\author{Bhargav Bhatt}
\address{Department of Mathematics \\ University of Michigan \\ Ann Arbor 48109 \\ USA}
\email{bhattb@umich.edu}
\begin{abstract}
We prove that the coherent cohomology of a proper morphism of noetherian schemes can be made arbitrarily $p$-divisible by passage to proper covers (for a fixed prime $p$). Under some extra conditions, we also show that $p$-torsion can be killed by passage to proper covers. These results are motivated by the desire to understand rational singularities in mixed characteristic, and have applications in $p$-adic Hodge theory.
\end{abstract}
\maketitle

\section{Introduction}
\label{sec:intro}

Fix a prime $p$. In this paper, we study the following question in mixed and positive characteristic geometry: given a scheme $X$ and a class $\alpha \in H^n(X,\calO_X)$ for some $n > 0$, does there exist a ``cover'' $\pi:Y \to X$ such that $\pi^* \alpha$ is divisible by $p$? Of course, as stated, the answer is trivially yes: take $Y$ to be a disjoint union of opens occurring in a \v{C}ech cocyle representing $\alpha$. However, the question becomes interesting if we impose geometric conditions on $\pi$, such as properness. The first obstruction encountered is the potential non-compactness of $X$: passage to proper covers cannot make cohomology classes $p$-divisible for the simplest of open varieties (such as $\A^2_{\F_p} - \{0\}$, see Example \ref{ex:properneeded}).  Our main result is that this is the only obstruction. In fact, we affirmatively answer the {\em relative} version of the above question for proper maps:

\begin{theorem}
\label{mixedcharpdiv}
Let $f:X \to S$ be a proper morphism of noetherian schemes with $S$ affine. Then there exists an alteration $\pi:Y \to X$ such $\pi^*(H^i(X,\calO_X)) \subset p(H^i(X,\calO_X))$ for $i > 0$.
\end{theorem}

Theorem \ref{mixedcharpdiv} is non-trivial only when $p$ is not invertible on $S$. When $p = 0$ on $S$, Theorem \ref{mixedcharpdiv} says that alterations kill the higher relative cohomology of the structure sheaf for proper maps, which is one of the main theorems of \cite{ddscposchar}. However, the techniques used in {\em loc.\ cit.} depend heavily on the use of Frobenius, and thus do not transfer to the mixed characteristic world. Our proof of Theorem \ref{mixedcharpdiv} is geometric --- we crucially use ideas from geometric class field theory and de Jong's work on stable curve fibrations --- and can be used to give a new proof of \cite[Theorem 1.5]{ddscposchar}.

The results of \cite{ddscposchar} are quite useful in the study of singularities in positive characteristic, and we expect that Theorem \ref{mixedcharpdiv} will be similarly useful once the study of singularities in mixed characteristic acquires more maturity. Moreover, Theorem \ref{mixedcharpdiv}  (as well as a stronger statement that we can prove when $\dim(S) \leq 1$, see Remark \ref{rmk:ptorsionkill}) has found surprising applications recently in $p$-adic Hodge theory: Beilinson's recent new and simple approach to Fontaine's $p$-adic comparison conjectures  (see \cite{Beilinsonpadic,Beilinsoncrystalline} and also \cite{Bhattpadicddr}) uses Theorem \ref{mixedcharpdiv} as the key geometric ingredient in the proof of the so-called $p$-adic Poincare lemma. A generalisation of Theorem \ref{mixedcharpdiv} that we hope to prove (see Remark \ref{rmk:findfinitecov}) would help extend these $p$-adic comparison theorems to the {\em relative} setting, and also have purely algebraic applications (see Remark \ref{rmk:splinterdsplinter}).

\subsection*{Outline of the proof of Theorem \ref{mixedcharpdiv}}
Assume first that $f$ has relative dimension $1$. If $S$ was a point, then a natural strategy is: replace $X$ with its normalisation, identify the group $H^1(X,\calO_X)$ with the tangent space to the Picard variety $\Pic^0(X)$ at the origin, and construct maps of curves such that the pullback on Picard varieties is divisible by $p$, at least at the expense of extending the ground field (these maps can be constructed via an old trick from geometric class field theory). For a non-trivial family of curves, the preceding argument can be applied to solve the problem over the generic point. Using the existence of compact moduli spaces of stable curves (or, even better, stable {\em maps}),  one can then extend the generic solution to one over an alteration of $S$. This is not quite enough as the alteration is no longer affine, but it reduces the theorem for morphisms of relative dimension $\leq 1$ to the theorem for alterations.  In general, theorems of de Jong show that an arbitrary proper morphism $f$ of relative dimension $d$ can be altered into a sequence of $d$ iterated stable curve fibrations over an alteration of the base. The previous argument then lets us inductively reduce the general problem to that for alterations. For this last case, we carefully fibre $S$ itself by curves while preserving certain cohomological properties, and proceed by induction on $\dim(S)$ (with $\dim(S) = 1$ being trivial).

\subsection*{Organisation of this paper}
Theorem \ref{mixedcharpdiv} is proven in \S \ref{sec:mainthmmixedchar}: we discuss a reduction to relative dimension $0$ in \S \ref{subsec:reducetord0}, and then prove this case in \S \ref{subsec:mixedcharrd0}. Note that when $\dim(S) \leq 1$, the latter step is unnecessary. In \S \ref{sec:mixedimpliespos}, we explain how to deduce the apparently stronger sounding \cite[Theorem 1.5]{ddscposchar} from Theorem \ref{mixedcharpdiv}.

\subsection*{Acknowledgements}
This paper was a part of the author's dissertation supervised by Johan de Jong, and would have been impossible without his support and generosity: many ideas here were discovered in conversation with de Jong. In addition, the intellectual debt owed to \cite{dJAlt,dJAlt2} is obvious and great.

\section{The main theorem}
\label{sec:mainthmmixedchar}
In order to flesh out the outline from \S \ref{sec:intro}, we first make the following trivial observation:

\begin{lemma}
\label{lem:reducetoexc}
If Theorem \ref{mixedcharpdiv} is true for excellent schemes $S$, then it is true in general.
\end{lemma}
\begin{proof}
Let $f:X \to S$ be a proper morphism of noetherian schemes. Then $\bigoplus_{i > 0} H^i(X,\calO_X)$ is a finite $\calO_S$-module, so it suffices to construct alterations of $X$ that make any fixed class $\alpha \in H^n(X,\calO_X)$ (for $n > 0$) divisible by $p$. For a fixed $\alpha$, the quadruple $(X,S,f,\alpha)$ can be approximated by a quadruple $(X_i,S_i,f_i,\alpha_i)$ with $X_i$ and $S_i$ excellent. By assumption, there is an alteration $\pi_i:Y_i \to X_i$ such that $\pi^*(\alpha_i) \in p (H^n(Y_i,\calO_{Y_i}))$. An irreducible component of $Y_i \times_{X_i} X$ dominating $X$ then does the job.
\end{proof}

Lemma \ref{lem:reducetoexc} allows us to restrict to excellent schemes in the sequel, which is convenient as we often want to normalise. We now make the following definition, integral to the rest of \S \ref{sec:mainthmmixedchar}:

\begin{definition}
Given a scheme $S$, we say that Condition $\calC_d(S)$ is satisfied if $S$ is excellent, and the following is satisfied by each irreducible component $S_i$ of $S$: given a proper surjective morphism $f:X \to S_i$ of relative dimenson $d$ with $X$ integral, there exists an alteration $\pi:Y \to X$ such that, with $g = f \circ \pi$, we have $\pi^*(\R^if_*\calO_X) \subset p (\R^i g_*\calO_Y)$ for $i > 0$.
\end{definition}

Theorem \ref{mixedcharpdiv} amounts to verifying Condition $\calC_d(S)$ for all excellent schemes $S$ (by Lemma \ref{lem:reducetoexc}). This verification is carried out in the sequel. More precisely, in \S \ref{subsec:reducetord0}, we will show that the validity of $\calC_0(S)$ for all excellent base schemes $S$ implies the validity of $\calC_d(S)$ for all integers $d$ and all excellent schemes $S$. We then proceed to verify Condition $\calC_0(S)$ in \S \ref{subsec:mixedcharrd0}.

\subsection{Reduction to the case of relative dimension $0$}
\label{subsec:reducetord0}

The objective of the present section is to show the relative dimension of maps considered in Theorem \ref{mixedcharpdiv} can be brought down to $0$ using suitable curve fibrations. The necessary technical help is provided by the following result, essentially borrowed from \cite{dJAlt2}, on extending maps between semistable curves.

\begin{proposition}
\label{extmapofsscurve}
Fix an integral excellent base scheme $B$ with generic point $\eta$. Assume we have semistable curves $\phi:C \to B$ and $\phi'_\eta:C'_\eta \to \eta$, and a $B$-morphism $\pi_\eta:C'_\eta \to C$. If $C'_\eta$ is geometrically irreducible, then we can alter $B$ to extend $\pi_\eta$ to a map of semistable cures over $B$, i.e., there exists an alteration $\tilde{B} \to B$ such that $C'_\eta \times_B \tilde{B}$ extends to a semistable curve over $\tilde{C'} \to \tilde{B}$ with $\tilde{C'}$ integral, and the map $\pi_\eta \times_B \tilde{B}$ extends to a $\tilde{B}$-map $\tilde{\pi}:\tilde{C'} \to C \times_B \tilde{B}$.
\begin{proof}
We may extend $C'_\eta$ to a proper $B$-scheme using the Nagata compactification theorem (see \cite[Theorem 4.1]{ConradNagata}). By taking the closure of the graph of the rational map defined from this compactification to $C$ by $\pi_\eta$, we obtain a proper dominant morphism $\phi':C' \to B$ of integral schemes whose generic fibre is the geometrically irreducible curve $\phi'_\eta:C'_\eta \to B$, and a $B$-map $\pi:C' \to C$ extending $\pi_\eta:C'_\eta \to C$. The idea, borrowed from \cite[\S 4.18]{dJAlt}, is the following: modify $B$ to make the strict transform of $C' \to B$ flat, alter the result to get enough sections which make the resulting datum generically a stable curve, use compactness of the moduli space of stable curves to extend the generically stable curve to a stable curve after further alteration, and then use stability and flatness to get a well-defined morphism from the resulting stable curve to the original one extending the existing one over the generic point. Instead of rewriting the details here, we refer the reader to \cite[Theorem 5.9]{dJAlt2} which directly applies to $\phi'$ to finish the proof (the integrality of $\tilde{C'}$ follows from the irreducibility of the generic fibre $C'_\eta \times_B \tilde{B}$).
\end{proof}
\end{proposition}

\begin{remark}
Proposition \ref{extmapofsscurve}, while sufficient for the application we have in mind, is woefully inadequate in terms of the permissible generality. Similar ideas can, in fact, be used to show something much better: for any flat projective morphism $X \to B$, there exists an ind-proper algebraic stack $\overline{\calM}_g(X) \to B$ parametrising $B$-families of stable maps from genus $g$ curves to $X$, see \cite{AbramovichOort}.
\end{remark}

In addition to constructing maps of semistable curves, we will also need to construct maps that preserve sections. The following lemma says we can do so at a level of generality sufficient for our purposes.

\begin{lemma}
\label{extsectsscurve}
Fix an integral excellent base scheme $B$, two semistable curves $\phi_1:C_1 \to B$ and $\phi_2:C_2 \to B$,  and a surjective $B$-map $\pi:C_2 \to C_1$. Then any section of $\phi_1$ extends to a section of $\phi_2$ after an alteration of $B$, i.e., given a section $s:B \to C_1$, there exists an alteration $b:\tilde{B} \to B$ such that the induced map $\tilde{B} \to B \to C_1$ factors through a map $\tilde{B} \to C_2$.
\end{lemma}
\begin{proof}
Let $\eta$ be the generic point of $B$, let $s:B \to C_1$ be the section of $\phi_1$ under consideration, and let ${s}_\eta:\eta \to C_1$ denote the restriction of $s$ to the generic point. By the surjectivity of $\pi$, the map $\pi_\eta:(C_2)_\eta \to (C_1)_\eta$ is surjective. Thus, there exists a finite surjective morphism $\eta' \to \eta$ such that the induced map $\eta' \to C_1$ factors through some map $s'_\eta:\eta' \to C_2$. If $B'$ denotes the normalisation of $B$ in $\eta' \to \eta$, then the map ${s}'_\eta$ spreads out to give a rational map $B' \dashrightarrow C_2$. Taking the closure of the graph of this rational map (over $B$) gives an alteration $b:\tilde{B} \to B$ such that the induced map $\tilde{B} \to C_1$ factors through a map $\tilde{s_2}:\tilde{B} \to C_2$, proving the claim.
\end{proof}

Proposition \ref{extmapofsscurve} lets us to construct maps of semistable curves by constructing them generically. We now construct the desired maps generically; the idea of this construction belongs to class field theory.

\begin{lemma}
\label{pdivcohcurvefield}
Let $X$ be a proper curve over a field $k$. Then there exists a field extension $k'$ of $k$, a proper smooth curve $Y$ over $k'$ with geometrically irreducible connected components, and a finite flat map $\pi:Y \to X_{k'}$ such that the induced map $\pi^*:\Pic(X_{k'}) \to \Pic(Y)$ of fppf sheaves of abelian groups on $\Spec(k)$ is divisible by $p$ in $\Hom(\Pic(X),\Pic(Y))$. 
\end{lemma}
\begin{proof}
The statement to be proven is stable under taking finite covers of $X$, and can be proven one connected component at a time. Thus, after picking a suitable finite extension $k'$ of $k$ and normalising $X_{k'}$, we may assume that $X$ is a smooth projective geometrically connected curve of genus $\geq 1$ with a rational point $x_0 \in X(k)$. The point $x$ defines the Abel-Jacobi map $X \to \Pic^0(X) \subset \Pic(X)$ via $x \mapsto \calO([x]) \otimes \calO(-[x_0])$. Riemann-Roch implies that this map is a closed immersion. We set $\pi:Y \to X$ to be the normalised inverse image of $X$ under the multiplication by $p$ map $[p]:\Pic(X) \to \Pic(X)$. It follows that the pullback $\pi^*:\Pic(X) \to \Pic(Y)$ factors through multiplication by $p$ on $\Pic(X)$ and is therefore divisible by $p$.
\end{proof}

\begin{remark}
Lemma \ref{pdivcohcurvefield} and the discussion below use basic properties of the relative Picard scheme of a proper flat family $f:X \to S$ of curves. A general reference for this object is \cite[\S 8-9]{NeronBook}. In this paper, we define $\Pic(X/S)$ as the fppf sheaf $\R^1 f_* \G_m$ on the category of all $S$-schemes. If $f$ has a section, then one can identify $\Pic(X/S)(T) \simeq \Pic(X \times_S T)/\Pic(T)$.  Since $f$ has relative dimension $1$, deformation theory implies that $\Pic(X/S)$ is smooth (as a functor). Two additional relevant properties are: (a) if $f$ has geometrically reduced fibres, then $\Pic(X/S)$ is representable by a smooth group scheme (by Artin's work), and (b) if $f$ is additionally semistable, then the connected component $\Pic^0(X/S)$ is semi-abelian.
\end{remark}

Lemma \ref{pdivcohcurvefield} allows us to construct covers of semistable curves that generically induce a map divisible by $p$ on cohomology. We now show how to globalise this construction; this forms one of the primary ingredients of our proof of Theorem \ref{mixedcharpdiv}.

\begin{proposition}
\label{pdivcohcurve}
Let $\phi:X \to T$ be a projective family of semistable curves with $T$ integral and excellent. Then there exists a diagram
\[ \xymatrix{ \tilde{X} \ar[r]^{\pi} \ar[d]^{\tilde{\phi}} & X \ar[d]^\phi \\ \tilde{T} \ar[r]^\psi & T } \]
satisfying the following:
\begin{enumerate}
\item The scheme $\tilde{T}$ is integral, and the map $\psi$ is an alteration.
\item $\tilde{\phi}$ is a projective family of semistable curves, and the map $\pi$ is proper and surjective.
\item The pullback map $\psi^*\R^1 \phi_*\calO_X \to \R^1\tilde{\phi}_*\calO_{\tilde{X}}$ is divisible by $p$ in $\Hom(\psi^*\R^1 \phi_*\calO_X, \R^1\tilde{\phi}_*\calO_{\tilde{X}})$.
\end{enumerate}
\end{proposition}
\begin{proof}
For any family $\phi:X \to T$ of projective semistable curves, there is a natural identification of $\R^1 \phi_*\calO_X$ with the normal bundle of the zero section of the semiabelian scheme $\Pic^0(X/T) \to T$. Moreover, given another semistable curve $\tilde{\phi}:\tilde{X} \to T$ and a morphism of semistable curves $\pi:\tilde{X} \to X$ over $T$, the induced map $\R^1(\pi): \R^1\phi_*\calO_X \to \R^1\tilde{\phi}_*\calO_{\tilde{X}}$ can be identified as the map on the corresponding normal bundles at $0$ induced by the natural morphism $\Pic^0(\pi):\Pic^0(X/T) \to \Pic^0(\tilde{X}/T)$. As multiplication by $n$ on smooth commutative $T$-group schemes induces multiplication by $n$ on the normal bundles at $0$, if $\Pic^0(\pi)$ is divisible by $p$, so is  $\R^1(\pi)$. As the formation of $\R^1\phi_*\calO_X$ commutes with arbitrary base change on $T$, it suffices to show:  there exists an alteration $\psi:\tilde{T} \to T$ and a morphism of semistable curves $\pi:\tilde{X} \to X \times_T \tilde{T}$ over $\tilde{T}$ such that the induced map $\Pic^0(\pi)$ is divisible by $p$.  Our strategy will be to construction a solution to this problem generically on $T$, and then use Proposition \ref{extmapofsscurve} and some elementary properties of semiabelian schemes to globalise.

Let $\eta$ denote the generic point of $T$. By Lemma \ref{pdivcohcurvefield}, we can find a finite extension $\eta' \to \eta$, and a proper smooth curve $Y_{\eta'} \to \eta'$ with geometrically irreducible components such that the induced map $\Pic^0(X_{\eta'}) \to \Pic^0(Y_{\eta'})$ is divisible by $p$. After replacing the map $X \to T$ with its base change along the normalisation of $T$ in $\eta' \to \eta$, we may assume that $\eta' = \eta$. The situation so far is summarised in the diagram
\[ \xymatrix{ \sqcup_i {Y_\eta}_i = Y_\eta  \ar[d] \ar[r] & X \ar[d] \\ \eta \ar[r] & T } \]
where the ${Y_\eta}_i$ are the (necessarily) geometrically irreducible components of $Y_\eta$. As each of the ${Y_\eta}_i$ is smooth as well, we may apply Proposition \ref{extmapofsscurve} to extend each ${Y_\eta}_i$ to a semistable curve $Y_i \to T_i$ where $T_i \to T$ is some alteration of $T$, such that the map ${Y_\eta}_i \to X$ extends to a map $Y_i \to X$. Setting $\tilde{T}$ to be a dominating irreducible component of the fibre product of all the $T_i$ over $T$, and setting $\tilde{X}$ to be the disjoint of $Y_i \times_{T_i} \tilde{T}$, we find the following: an alteration $\tilde{T} \to T$, a semistable curve $\tilde{X} \to \tilde{T}$ extending $Y_\eta \times_T \tilde{T}$, and a map $\tilde{\pi}:\tilde{X} \to X$ extending the existing one over the generic point. We will now check the required divisibility.

 As explained earlier, we must show that the resulting map $\Pic^0(X \times_T \tilde{T}/\tilde{T}) \to \Pic^0(\tilde{X}/\tilde{T})$ is divisible by $p$. This divisibility holds generically on $\tilde{T}$ by construction. Next,  note that the relative $\Pic^0$ of any semistable curve is a semiabelian group scheme. The normality of $\tilde{T}$ implies that restriction to the generic point is a fully faithful functor from the category of semiabelian schemes over $\tilde{T}$ to the analogous category over its generic point (see \cite[Proposition I.2.7]{FaltingsChai}). In particular, the generic divisibility by $p$ ensures the global divisibility by $p$, proving the existence of $\tilde{X}$ with the desired properties. 
\end{proof}

Recall that our immediate goal is to reduce Theorem \ref{mixedcharpdiv} to verifying Condition $\calC_0(S)$. Proposition \ref{pdivcohcurve} lets us make the relative cohomology of a curve fibration divisible by $p$ on passage to alterations, while de Jong's theorems let us alter an arbitrary proper dominant morphism into a tower of curve fibrations over an alteration of the base. These two ingredients combine to yield the promised reduction in relative dimension.

\begin{proposition}
\label{reducereldim}
Let $S$ be an excellent scheme such that Condition $\calC_0(S)$ is satisfied. Then $\calC_d(S)$ is satisfied for all $d \geq 0$.
\end{proposition}

\begin{proof}
As Condition $\calC_d(S)$ is defined in terms of the irreducible components of $S$, we may assume that $S$ is integral itself. Fix integers $d,i > 0$, an integral scheme $X$, and a proper surjective morphism $f:X \to S$ of relative dimension $d$. By Corollary 5.10 of \cite{dJAlt2}, after replacing $X$ by an alteration, we may assume that $f:X \to S$ factors as follows:
\[ \xymatrix{ X \ar[rd]^f \ar[d]^\phi & \\
			 T \ar[r]^{f'} & S } \]
Here $\phi$ is a projective semistable curve, and $f'$ is a proper surjective morphism of integral excellent schemes of relative dimension $d-1$. Also, at the expense of altering $T$ further, we may assume that $\phi$ has a section $s:T \to X$. The fact that $\phi$ is a semistable curve gives us the formula $\calO_T \simeq \phi_*\calO_X$. Using the section $s$ and the Leray spectral sequence, we find an exact sequence
\[0 \to \R^if'_*\calO_T \to \R^if_*\calO_X \to \R^{i-1}f'_*\R^1\phi_*\calO_X \to 0\]
that is naturally split by the section $s$. Our strategy will be to prove divisibility for $\R^if_*\calO_X$ by working with the two edge pieces occuring in the exact sequence above. In more detail, we apply the inductive hypothesis to choose an alteration $\pi':T' \to T$ such that, with $g' = f'\circ\pi'$, we have $\pi'^*(\R^if'_*\calO_T) \subset p(\R^ig'_*\calO_{T'})$. The base change of $\phi$ and $s$ along $\pi'$ define for us a diagram
\[ \xymatrix{ X' = X \times_T T' \ar[r]^-{\pr_1} \ar[d]^{\phi'} & X \ar[rd]^f \ar[d]^\phi & \\
			 T' \ar[r]^{\pi'} \ar@/^1pc/[u]^{s'} & T \ar[r]^{f'} \ar@/^1pc/[u]^{s} & S } \]
The commutativity of the preceding diagram gives rise to a morphism of exact sequences
\[\xymatrix{ 0 \ar[r] & \R^if'_*\calO_T \ar[r] \ar[d]^{\pi'^*} & \R^if_*\calO_X \ar@/_1pc/[l]_{s^*} \ar[r] \ar[d]^{\pr_2^*} & \R^{i-1} f'_*\R^1\phi_*\calO_X \ar[d]^{\R^1\pr_2^*} \ar[r] & 0 \\
   0 \ar[r] & \R^ig'_*\calO_{T'} \ar[r] & \R^i (f \circ \pr_1)_*\calO_{X'} \ar[r] \ar@/_1pc/[l]_{s'^*} & \R^{i-1}g'_*\R^1\phi'_*\calO_{X'} \ar[r] & 0  } \]
compatible with the exhibited splittings. The map $\phi'$ is a semistable curve with a section $s'$. Applying Proposition \ref{pdivcohcurve} and using Lemma \ref{extsectsscurve}, we can find a commutative diagram
\[ \xymatrix{ X'' \ar[r]^-a \ar[d]^{\phi''} &  X' = X \times_T T' \ar[r]^-{\pr_1} \ar[d]^{\phi'} & X \ar[rd]^f \ar[d]^\phi & \\
			  T'' \ar[r]^{\pi''} \ar@/^1pc/[u]^{s''} &  T' \ar[r]^{\pi'} \ar@/^1pc/[u]_{s'} & T \ar[r]^{f'} \ar@/^1pc/[u]^{s} & S } \]
where $\pi''$ is an alteration, $\phi''$ is a semistable curve, $a$ is an alteration, $s''$ is a section of $\phi''$ (compatible with $s'$ and $s$ thanks to the commutativity of the picture), such that $a^*\R^1\phi'_*\calO_{X'} \to \R^1\phi''_*\calO_{X''}$ is divisible by $p$. Setting $g'' = g' \circ \pi''$ gives a diagram of exact sequences
\[\xymatrix{ 0 \ar[r] & \R^if'_*\calO_T \ar[r] \ar[d]^{\pi'^*} & \R^if_*\calO_X  \ar@/_1pc/[l]_{s^*} \ar[r] \ar[d]^{\pr_2^*} & \R^{i-1} f'_*\R^1\phi_*\calO_X \ar[d]^{\R^1\pr_2^*} \ar[r] & 0 \\
   0 \ar[r] & \R^ig'_*\calO_{T'} \ar[r] \ar[d]^{\pi''^*} & \R^i (f \circ \pr_1)_*\calO_{X'} \ar@/_1pc/[l]_{s'^*} \ar[r] \ar[d]^{a^*} & \R^{i-1}g'_*\R^1\phi'_*\calO_{X'} \ar[r] \ar[d]^{\R^1a^*} & 0  \\
   0 \ar[r] & \R^ig''_*\calO_{T''} \ar[r] & \R^i (f \circ \pr_1 \circ a)_*\calO_{X''} \ar[r] \ar@/_1pc/[l]_{s''^*}  & \R^{i-1}g''_*\R^1\phi''_*\calO_{X''} \ar[r] & 0 } \]
which is compatible with the exhibited splittings of each sequence. As $\R^1a^*$ is divisible by $p$, the image of right vertical composition is divisible by $p$. The image of the left vertical composition is divisible by $p$ by construction of $\pi'$. By compatibility of the morphism of exact sequences with the exhibitted splittings, the image of the middle vertical composition is also divisible by $p$. Replacing $X''$ be an irreducible component dominating $X$ then proves the claim.
\end{proof}

\begin{remark}
	\label{rmk:ptorsionkill}
	Consider the special case of Theorem \ref{mixedcharpdiv} when the base $S$ has dimension $\leq 1$; for example, $S$ could be the spectrum of an excellent discrete valuation ring. Any alteration of such an $S$ can be dominated by a finite cover of $S$, so $\calC_0(S)$ is trivially satisfied. Proposition \ref{reducereldim} then already implies that Theorem \ref{mixedcharpdiv} is true for such $S$. In fact, tracing through the proof (and using the strong $p$-divisibility in Proposition \ref{pdivcohcurve} (3)), one observes that a stronger statement has been shown: for any proper morphism $f:X \to S$, there is a proper surjective morphism $\pi:Y \to X$ such that, with $g = f \circ \pi$, the pullback $\pi^*: \tau_{\geq 1} \R f_* \calO_X \to \tau_{\geq 1} \R g_* \calO_Y$ induces the $0$ map on $- \otimes^\L_{\Z} \Z/p$. Concretely, this means: in addition to making higher cohomology classes $p$-divisible on passage to alterations, one can also kill $p$-torsion classes by alterations. It is this stronger statement that has applications in $p$-adic Hodge theory. We hope in the future to extend this stronger conclusion to higher dimensional base schemes $S$. 
\end{remark}

\begin{remark}
\label{rmk:splinterdsplinter}
The stronger statement discussed at the end of Remark \ref{rmk:ptorsionkill} has purely algebraic consequences: it lifts \cite[Theorem 1.4]{ddscposchar} to $p$-adically complete noetherian schemes (after a small extra argument). In particular, it implies that splinters over $\Z_p$ have rational singularities after inverting $p$. Such a statement is interesting from the perspective of the direct summand conjecture (see \cite{HochsterSurvey}) as there are {\em no} known non-trivial restrictions on a splinters in mixed characteristic (to the best of our knowledge).
\end{remark}

\subsection{The case of relative dimension $0$}
\label{subsec:mixedcharrd0}

In this section we will verify Condition $\calC_0(S)$ for all excellent schemes $S$. After unwrapping definitions and some easy reductions, one reduces to showing: given an alteration $f:X \to S$ with $S$ affine and a class $\alpha \in H^i(X,\calO)$ with $i > 0$, there exists an alteration $\pi:Y \to X$ such that $p \mid \pi^*(\alpha)$. If $\alpha$ arose as the pullback of a class under a morphism $X \to \overline{X}$ with $\overline{X}$ proper over an affine base of dimension $\dim(S) - 1$, then we may conclude by induction using Proposition \ref{reducereldim}. The proof below will show that, at the expense of certain technical but manageable modifications, this method can be pushed through; the basic geometric ingredient is Lemma \ref{geompresentlemma}. The main result is:

\begin{proposition}
\label{killbyalt}
The Condition $\calC_0(S)$ is satisfied by all excellent schemes $S$.
\end{proposition}

Our proof of Proposition \ref{killbyalt} will consist of a series of reductions which massage $S$ until it becomes a geometrically accessible object (see Lemma \ref{reducetomixedchar} for the final outcome of these ``easy'' reductions).

\begin{warning}
\label{abusewarnmixedchar}
For conceptual clarity, we often commit the following abuse of mathematics in the sequel: when proving a statement of the form that $\calC_d(S)$ is satisfied for all integers $d$ and a particular scheme $S$, we ignore the restrictions on integrality and relative dimension imposed by Condition $\calC_d(S)$ while making certain constructions; the reader can check that in each case the statement to be proven follows from our constructions by taking suitable irreducible components (see Lemma \ref{reducetolocal} for an example). We strongly believe that this abuse, while easily fixable, enhances readability.
\end{warning}

We first observe that the problem is Zariski local on $S$.

\begin{lemma}
\label{reducetolocal}
The Condition $\calC_d(S)$ is local on $S$ for the Zariski topology, i.e., if $\{U_i \hookrightarrow X\}$ is a Zariski open cover of $X$, then $\calC_d(S)$ is satisfied if and only if $\calC_d(U_i)$ is satisfied  for all $i$.
\end{lemma}
\begin{proof}
We will first show that $\calC_d(S)$ implies $\calC_d(U)$ for any open $j:U \to S$. By Nagata compactification (see \cite[Theorem 4.1]{ConradNagata}), given any alteration $f:X \to U$, we can find an alteration $\overline{f}:\overline{X} \to S$ extending $f$ over $U$. As $j:U \to S$ is flat, we have that $j^*\R^i\overline{f}_*\calO_{\overline{X}} = \R^i f_*\calO_X$. By assumption, we can find an alteration $\overline{\pi}:\overline{Y} \to \overline{X}$ such that, with $\overline{g} = \overline{f} \circ \overline{\pi}$, we have $\overline{\pi}^*\R^i\overline{f}_*\calO_{\overline{X}} \subset p(\R^i\overline{g}_*\calO_{\overline{Y}})$. Restricting to $U$ and using flat base change for $\overline{g}$ produces the desired result. 

Conversely, assume there exists a cover $\{U_i \hookrightarrow X\}$ such that $\calC_d(U_i)$ is true. Given an alteration $f:X \to S$, define $f_i:X_{U_i} \to U_i$ to be the natural map. The assumption implies that we can find alterations $\pi_i:Y_i \to X_{U_i}$ such that, with  $g_i = f_i \circ \pi_i$, we have $\pi_i^*(\R^j {f_i}_* \calO_{X_{U_i}}) \subset p (\R^j {g_i}_*\calO_{Y_i})$ for each $i$. By an elementary spreading out argument (see \cite[Proposition 4.1]{ddscposchar}),  we can find $\pi:Y \to X$ such that $\pi \times_S U_i$ factors through $\pi_i$. As taking higher pushforwards commutes with restricting to open subsets,  we see that $\pi^*(\R^j f_*\calO_X) \subset \R^j g_*\calO_Y$ is a subsheaf that is locally inside $p(\R^j g_*\calO_Y)$. As containments between two subsheaves of a given sheaf can be detected locally, the claim follows. 
\end{proof}

Next, we show how to localise for the topology of finite covers.

\begin{lemma}
\label{reducetofcov}
The Condition $\calC_d(S)$ is local on $S$ for the topology of finite covers, i.e., if $g:S' \to S$ is finite surjective, then $\calC_d(S)$ is satisfied if and only if $\calC_d(S')$ is satisfied.
\end{lemma}
\begin{proof}
For the forward direction, we use the fact that a proper surjective map $f:X \to S'$ defines a proper surjective map $g\circ f:X \to S$ of the same relative dimension. For the converse direction, we use that given a proper surjective map $f:Y \to S$, the base change $Y \times_S S' \to Y$ is a finite surjective map with the additional property that $Y \times_S S'$ admits a proper surjective map to $S'$ of the same relative dimension as $f$. We omit the details.
\end{proof}

Combining the preceding observations, we show how to pass to quasifinite covers.

\begin{lemma}
\label{reducetoqfcov}
If $g:S' \to S$ is quasifinite and surjective and $\calC_d(S)$ is satisfied, then $\calC_d(S')$ is satisfied.
\end{lemma}
\begin{proof}
By Zariski's main theorem (Th\'eor\`eme $8.12.6$ of \cite{EGA4_3}), we can factor  $g$ as $S' \stackrel{j}{\to} \overline{S'} \stackrel{\overline{g}}{\to} S$ with $j$ an open immersion, and $\overline{g}$ finite surjective. Lemma \ref{reducetofcov} implies that $\calC_d({\overline{S'}})$ is satisfied. The first half of the proof of Lemma \ref{reducetolocal} then shows that $\calC_d(S')$ is also satisfied, as desired.
\end{proof}

Finally, we show how to \'etale localise:

\begin{lemma}
\label{reducetolocalet}
The Condition $\calC_d(S)$ is \'etale local on $S$, i.e., if $g:S' \to S$ is a surjective \'etale morphism, then $\calC_d(S)$ is satisfied  if and only if $\calC_d(S')$ is satisfied.
\end{lemma}
\begin{proof}
	If $\calC_d(S)$ is satisfied, then $\calC_d(S')$ is also satisfied by Lemma \ref{reducetoqfcov}. For the converse direction, using Lemma \ref{reducetolocal}, we may assume that $S$ and $S'$ are both local schemes. An observation of Gabber (see \cite[Lemma 2.1]{Bhattanngrpsch}) lets us  find a diagram
\[ \xymatrix{ \sqcup_i U_i \ar[r] \ar[d]^h & T \ar[d]^{\pi} \\ S' \ar[r]^g & S }\]
such that $\pi$ is finite surjective, $\sqcup U_i \to T$ forms a Zariski cover, and $h$ is some map. The commutativity of the diagram forces $h$ to be quasifinite, while the locality of $S'$ forces $h$ to be surjective. Since we are assuming that $\calC_d(S')$ is satisfied, Lemma \ref{reducetoqfcov} now implies that $\calC_d(\sqcup_i U_i)$ is satisfied. Using Lemma \ref{reducetolocal}, we deduce that $\calC_d(T)$ is satisfied. Lemma \ref{reducetofcov} then allows us to conclude that $\calC_d(S)$ is satisfied, as desired.
\end{proof}

Having \'etale localised, we prove an approximation result.

\begin{lemma}
\label{reducetoftoverZ}
The Condition $\calC_d(S)$ is satisfied by all excellent schemes $S$ if it is satisfied by all affine schemes $S$ of finite type over $\Z$.
\end{lemma}
\begin{proof}
By Proposition \ref{reducereldim}, we may restrict ourselves to $d = 0$ (the resulting simplification is purely notational). Assume that $\calC_d(S)$ is satisfied whenever $S$ is of finite type over $\Z$. Let $S'$ be an arbitrary excellent affine scheme, and let $f:X' \to S'$ be a proper morphism. Given a cohomology class $\alpha' \in H^i(X',\calO_{X'})$, the quadruple $(X',S',f',\alpha')$ is defined by a finite amount of algebraic data. Consequently, it can be obtained from a similar quadruple $(X,S,f,\alpha)$ along a base change $S' \to S$ where $S$ has finite type over $\Z$. By assumption, there exists an alteration $\pi:Y \to X$ such that $\pi^*{\alpha}$ is divisible by $p$. It follows then that a dominating irreducible component of $\pi':Y \times_S S' \to X'$ provides the desired alteration.
\end{proof}

Next, we complete at $p$.

\begin{lemma}
\label{reducetomixedchar}
The Condition $\calC_d(S)$ is satisfied by all affine excellent schemes $S$ if it is satisfied by all affine schemes $S$ of finite type over the Witt vector ring $\W(\overline{\F}_p)$.
\end{lemma}
\begin{proof}
We first summarise the idea informally. Using Lemmas \ref{reducetoftoverZ} and \ref{reducetolocalet}, one reduces to checking $\calC_0(S)$ for $S$ of finite type over the strict henselisation $R^\sh$, where $R = \Z_{(p)}$ is the localisation of $\Z$ at $p$. To deduce the statement over $R^\sh$ from that over $\widehat{R^\sh} = \W(\overline{\F}_p)$, we use Popescu's approximation theorem which allows us to write $\widehat{R^\sh}$ is an inductive limit of smooth $R^\sh$-algebras. The crucial point here is that any smooth $R^\sh$-algebra with non-empty special fibre has an $R^\sh$-valued point; the details follow.

By Proposition \ref{reducereldim} and Lemma \ref{reducetoftoverZ}, it suffices to show that $\calC_0(S)$ is satisfied whenever $S$ is affine and of finite type over $\Z$. There is nothing to show when $p \in \calO_S^\ast$. By  Nagata compactification (see \cite[Theorem 4.1]{ConradNagata}), it suffices to check $\calC_0(S)$ for $S$ affine and of finite type over $R = \Z_{(p)}$, the localisation of $\Z$ at $p$. By Lemma \ref{reducetolocalet} and a limit argument, we reduce to verifying $\calC_0(S)$ for $S$ affine and of finite type over $R^\sh$. Let $S$ be such a scheme, and let $f:X \to S$ be an alteration of $S$. As $\bigoplus_{i > 0} H^i(X,\calO_X)$ is a finite $\calO_S$-module, we may work one class at a time. Let $\alpha \in H^i(X,\calO_X)$ be a cohomology class of degree $i > 0$. By assumption, there exists an alteration $\widehat{\pi}:\widehat{Y} \to X_{\widehat{R^\sh}}$ such that $p | \widehat{\pi}^*(\alpha)$, where $\widehat{R^\sh} \simeq \W(\overline{\F}_p)$ is the $p$-adic completion of $R^\sh$. By the main theorem of \cite{PopescuNeronDesing}, the completion $R^\sh \to \widehat{R^\sh}$ is an ind-smooth morphism, i.e., one can write $\widehat{R^\sh} = \colim R_i$, where $R^\sh \to R_i$ is finite type and smooth. Furthermore, since $R^\sh \to \widehat{R^\sh}$ has a non-empty special fibre, the same is true for $R^\sh \to R_i$ for each $i$. By virtue of everything being of finite presentation, there exists an index $i$ and an alteration $\pi_i:Y_i \to X_{R_i}$ giving $\widehat{\pi}$ over $\widehat{R^\sh}$ such that $p | \pi_i^*(\alpha)$. As $R^\sh \to R_i$ is smooth with non-empty special fibre, we can cut $R_i$ with appropriately chosen hyperplane sections to find a quotient $R_i \to T$ such that the composite $R^\sh \to R_i \to T$ is finite \'etale. In fact, $R^\sh \simeq T$ since $R^\sh$ is strictly henselian and $T$ is local.  The fibre $\pi_i \times_{R_i} T$ is then easily seen to do the job.
\end{proof}

We have reduced the proof of Theorem \ref{mixedcharpdiv} to showing Condition $\calC_0(S)$ for affine schemes $S$ of finite type over $B = \Spec(\W(\overline{\F}_p))$. Given an alteration of such an $S$, the non-quasi-finite locus of the alteration is a closed subset $Z \subset S$ of codimension $\geq 2$ that is often called the {\em center} of the alteration. Our strategy for proving Theorem \ref{mixedcharpdiv} is to construct, at the expense of localising a little on $S$, a partial compactification $S \hookrightarrow \overline{S}$ with $\overline{S}$ proper over a lower dimensional base such that $Z$ remains closed in $\overline{S}$. This last condition ensures that the alteration in question can be extended to an alteration of $\overline{S}$ without changing the center. As the center has not changed, the cohomology of the newly created alteration maps onto that of the older alteration, thereby paving the way for an inductive argument via Proposition \ref{reducereldim}. The precise properties needed to carry out the above argument are ensured by the presentation lemma that follows.

\begin{lemma}
\label{geompresentlemma}
Let $B = \Spec(\W(\overline{\F}_p))$ be the maximal unramified extension of $\Z_p$. Let $\widehat{S}$ be a local,  flat,  and essentially finitely presented $B$-scheme of relative dimension $\geq 1$. Given a closed subset $\widehat{Z} \subset \widehat{S}$ of codimension $\geq 2$, we can find a diagram of $B$-schemes of the form
\[ \xymatrix{ s \in Z \ar[r]^{i_Z} \ar[rrd] & S \ar[rd] \ar[r]^j & \overline{S} \ar[d]^{\pi} & \partial\overline{S} \ar[ld] \ar[l]_i \\
			& & W & } \]
satisfying the following:
\begin{enumerate}
\item All the schemes in the diagram above are of finite type over $B$.
\item $S$ is an integral scheme, $i_Z$ is a closed subscheme, $s$ is a closed point, and the germ of $i_Z$ at $s$ agrees with $\widehat{Z} \subset \widehat{S}$.
\item $i$ is the inclusion of a Cartier divisor, and $j$ is the open dense complement of $i$.
\item $W$ is an integral affine scheme with $\dim(W) = \dim(S) - 1$.
\item $\pi$ is proper, $\pi\mid_S$ is affine, and both these maps have fibres of equidimension $1$.
\item $\pi\mid_Z$ and $\pi\mid_{\partial\overline{S}}$ are finite. In particular, $j(i_Z(Z))$ is closed in $\overline{S}$.
\end{enumerate}
\end{lemma}
\begin{proof}
We first remark that since $\widehat{S}$ is a local $B$-scheme, the residue field at the closed point of $\widehat{S}$ has positive characteristic.  We begin by choosing an ad hoc finite type model of $\widehat{Z} \hookrightarrow \widehat{S}$ over $B$, i.e., we find a map $i_Y:Y \to T$ and a point $y \in Y$  satisfying the following: the map $i_Y$ is a closed immersion of finite type integral $B$-schemes with codimension $\geq 2$, and the germ of $i_Y$ at $y$ is the given map $\widehat{Z} \hookrightarrow \widehat{S}$. Next, we choose an ad hoc compactification $T \hookrightarrow \overline{T}$ over $B$, i.e., $\overline{T}$ is a projective $B$-scheme containing $T$ as a dense open subscheme. By replacing $T$ with the complement of a suitable ample divisor missing the point $y$ in the special fibre (and hence in all of $\overline{T}$ by properness), we may assume that the complement $\partial\overline{T}$ is an ample divisor flat over $B$. We denote by $\overline{Y}$ the closure of $Y$ in $\overline{T}$, and by $\partial\overline{Y} = \overline{Y} - Y$ its boundary. As $Y$ had codimension $\geq 2$, its closure $\overline{Y}$ also has codimension $\geq 2$, while the boundary $\partial\overline{Y}$ has codimension $\geq 3$ as $Y$ is not contained in $\partial\overline{T}$. We will modify $T$ and $\overline{T}$ to eventually find the required $S$ and $\overline{S}$.

Let $d$ denote the relative dimension of $\overline{T}$ over $B$. By construction, this is the relative dimension of $\widehat{S}$ over $B$ as well. The next step is to find a finite morphism $\phi:\overline{T} \to \P^d$ such that $\phi(y) \notin \phi(\partial\overline{T})$. We find such a map by repeatedly projecting. In slightly more detail, say we have a finite morphism $\phi:\overline{T} \to \P^N$ for some $N > d$ such that $\phi(y) \notin \phi(\partial\overline{T})$. Then $\phi(\partial\overline{T})$ is a closed subscheme of codimension $\geq 2$. Moreover, by the flatness of $\partial\overline{T}$ over $B$, the same is true in the special fibre $\P^N \times_B \overline{\F}_p$. By basic facts of projective geometry in the special fibre, we can find a line $\ell$ through $\phi(y)$ that does not meet $\phi(\partial\overline{T})$. By the ampleness of $\partial\overline{T}$, this line cannot entirely be contained in $\phi(\overline{T})$. Thus, we can find a point on it that is not contained in $\phi(\overline{T})$. By projecting from this point, we see that we can find a finite morphism $\phi:\overline{T} \to \P^{N-1}$ such that $\phi(y) \notin \phi(\partial\overline{T})$. So far the discussion was taking place in the special fibre. However, by choosing a lift of this point to a $B$-point by smoothness of $\P^N$ and using the properness of $\partial\overline{T}$ to transfer the non-intersection condition from the special fibre to the total space, this construction can be made over $B$. Continuing this way, we can find a finite morphism $\phi:\overline{T} \to \P^d$ with the same property. As $\phi(\partial\overline{T})$ is now a very ample Cartier divisor, its complement $U \hookrightarrow \P^d$ is an open affine containing $\phi(y)$. We may now replace $T$ with $\phi^{-1}(U)$ and $Y$ with $\overline{Y} \cap \phi^{-1}(U)$ (this does not change the closure as $y \in \overline{Y} \cap \phi^{-1}(U)$ and $\overline{Y}$ is irreducible) to assume that we have produced the following: an algebraisation $i_Y:Y \to T$ of $\widehat{Z} \to \widehat{S}$ for some point $y \in Y$, a compactification $T \hookrightarrow \overline{T}$, and a finite morphism $\phi:\overline{T} \to \P^d$ such that  $T = \phi^{-1}(U)$ for some open affine $U \in \P^d$ that is the complement of a very ample divisor $H$, flat over $B$.

Now we project once more to obtain the desired curve fibration. As explained earlier, the closure $\overline{Y}$ has codimension $\geq 2$ in $\overline{T}$. Since we do not know that it is flat over $B$, the most we can say is that its image $\phi(\overline{Y})$ has codimension $\geq 1$ in the special fibre $\P^d \times_B \overline{\F}_p$. On the other hand, we know that $\partial\overline{T}$ was a $B$-flat divisor. Thus, its image $\phi(\partial\overline{T})$ also has codimension $\geq 1$ in the special fibre $\P^d \times_B \overline{\F}_p$. It follows that $\phi(\overline{Y} \cup \partial\overline{T})$ has codimension $\geq 1$ in the special fibre $\P^d \times_B \overline{\F}_p$. By choosing a closed point not in this image inside $U$ and lifting to a $B$-point as above, we find a $B$-point $p:B \to U \subset \P^d$ whose image does not intersect $\phi(\overline{Y} \cup \partial\overline{T})$. Projecting from this point gives rise to the following diagram:
\[ \xymatrix{ \Bl_{\phi^{-1}(p)}(T) \ar[r]^a \ar[d] & \Bl_{\phi^{-1}(p)}(\overline{T}) \ar[r]^-b \ar[d] & \Bl_p(\P^d) \ar[r]^-c \ar[d] & \P(T_p(\P^d)) \simeq \P^{d-1} \\
	T \ar[r] & \overline{T} \ar[r] & \P^d. & }\]
The horizontal maps enjoy the following properties: $c$ is a $\P^1$-fibration (in the Zariski topology), $b$ is a finite surjective morphism, and $a$ is an open immersion. In particular, the composite map $cb$ is a proper morphism with fibres of equidimension $1$. As the map $\phi:\overline{T} \to \P^d$ was chosen to ensure that $\phi^{-1}(U) = T$, the composite map $cba$ can be factored as 
\[\Bl_{\phi^{-1}(p)}(T) \to \Bl_{p}(U) \to \P^{d-1}. \]
The first map in this composition is finite surjective as $\phi$ is so, while the second map is an affine morphism with fibres of equidimension $1$ thanks to Lemma \ref{blowuplemma} below. It follows that the composite map $cba$ is an affine morphism with fibres of equidimension $1$. Lastly, by our choice of $p$, the map $cb$ restricts to a finite map on $\overline{Y}$ and $\overline{T}$ (here we identify subschemes of $\overline{T}$ not intersectng $\phi^{-1}(p)$ with those of the blowup). As explained earlier, the boundary $\partial\overline{Y}$ has codimension $\geq 3$ in $\overline{T}$. This implies that its special fibre has codimension $\geq 2$. Therefore, its image in $\P^{d-1}$ has codimension $\geq 1$. It follows that we can find an open affine $W \hookrightarrow \P^{d-1}$ not meeting the image of $\phi(\partial\overline{Y})$. Restricting the entire picture thus obtained to $W$, we find a diagram that looks like
\[ \xymatrix{ y \in Y \ar[r] \ar[rrd] & \Bl_{\phi^{-1}(p)}(T)_W \ar[r] \ar[rd] & \Bl_{\phi^{-1}(p)}(\overline{T})_W \ar[d]^\pi & \partial\overline{T}_W \ar[l] \ar[ld] \\
							& 									 & W.  & } \]
Setting $s = y$, $Z = Y$, $S = \Bl_{\phi^{-1}(p)}(T)_W$, $\overline{S} = \Bl_{\phi^{-1}(p)}(\overline{T})_W$, and $\partial\overline{S} = \partial\overline{T}_W$ implies the claim.
\end{proof}

The following elementary fact concerning blowups was used in Lemma \ref{geompresentlemma}:

\begin{proposition}
\label{blowuplemma}
Fix an affine regular base scheme $B$. Let $H \hookrightarrow \P^n \times B$ be a divisor that is flat and relatively ample over $B$, and let $U$ be the complement. For any point $p \in U(B)$, the blowup map $\Bl_p(U) \to \P(T_p(\P^n))$ is an affine morphism with fibres of equidimension $1$.
\end{proposition}
\begin{proof}
Let $b:\Bl_p(\P^n) \to \P^n$ be the blowup map, and let $\pi:\Bl_p(\P^n) \to \P(T_p(\P^n))$ be the morphism defined by projection. It is easy to see that $\pi$ is a $\P^1$-bundle. In fact, it can be identified with the projectivisation of the rank $2$ vector bundle $\calO(1) \oplus \calO$ on $\P(T_p(\P^n))$, with the exceptional divisor corresponding to the zero section of $\calO(1)$. As the ample divisor $H$ was disjoint from the center of the blowup, $b^*(H)$ defines an ample divisor on the fibres of $\pi$. By semicontinuity, for any vector bundle $\calE \in \Vect(\Bl_p(\P^n))$, the higher pushforwards $R^i\pi_*\calE(nH)$ vanish for $i > 0$ provided $n$ is sufficiently large. By the regularity of all schemes in sight, the same is true for any coherent sheaf. As $\Bl_p(U) \hookrightarrow \Bl_p(\P^n)$ is the complement of $b^*(H)$, the scheme $\pi|_{\Bl_p(U)}$ is affine by Serre's criterion. The claim about the fibre dimension follows from the observation that, geometrically, the fibre of $\pi|_{\Bl_p(U)}$ over a line $\ell$ passing through $p$, viewed as a point $[\ell] \in \P(T_p(\P^n))$, is simply $\ell \cap U$ which is a non-empty affine curve in $\ell$ thanks to the choice of $p$ and the positivity of $H$
\end{proof}

Before proceeding to the proof of Theorem \ref{mixedcharpdiv}, we record a cohomological consequence of certain geometric hypotheses. The hypotheses in question are the kind ensured by Lemma \ref{geompresentlemma}, while the consequences are those used in proof of Theorem \ref{mixedcharpdiv}.

\begin{proposition}
\label{cohvanprop}
Fix a noetherian scheme $\overline{X}$ of finite Krull dimension. Let $j:X \hookrightarrow \overline{X}$ be a dense open immersion whose complement $\Delta \subset \overline{X}$ is affine and the support of a Cartier divisor. Then $H^i(\overline{X},\calO) \to H^i(X,\calO)$ is surjective for all $i > 0$.
\end{proposition}
\begin{proof}
As $\Delta \subset \overline{X}$ is the support of a Cartier divisor, the complement $j$ is an affine map. This implies that 
\[j_*\calO_X \simeq \R j_*\calO_X.\]
Now consider the exact sequence
\[0 \to \calO_{\overline{X}} \to j_*\calO_X \to \calQ \to 0\]
where $\calQ$ is defined to be the cokernel. As $j_*\calO_X \simeq \R j_*\calO_X$, the middle term in the preceding sequence computes $H^i(X,\calO)$. By the associated long exact sequence on cohomology,  to show the claim, it suffices to show that $H^i(\overline{X},\calQ) = 0$ for $i > 0$. By construction, we have a presentation
\[j_*\calO_X = \colim_n \calO_{\overline{X}} (n\Delta).\]
Thus, we also have a presentation
\[ \calQ = \colim_n \calO_{\overline{X}}(n\Delta) / \calO_{\overline{X}}.\]
This presentation defines a natural increasing filtration $F^\bullet(\calQ)$ with
\[F^n(\calQ) = \calO_{\overline{X}}(n\Delta)/\calO_{\overline{X}}\]
for $n \geq 0$. The associated graded pieces of this filtration are
\[ \gr^n_F(\calQ) = \calO_{\overline{X}}(n\Delta) \otimes \calO_{\Delta}. \]
In particular, these pieces are supported on $\Delta$ which is an affine scheme by assumption. Consequently, these pieces have no higher cohomology. By a standard devissage argument, the sheaves $F^n(\calQ)$ have no higher cohomology for any $n$. Then $\calQ$ has no higher cohomology either (as cohomology commutes with filtered colimits of sheaves on noetherian schemes of finite Krull dimension), establishing the claim.
\end{proof}

We now have enough tools to finish proving Theorem \ref{mixedcharpdiv}.

\begin{proof}[Proof of main theorem]
Our goal is to show that Condition $\calC_0(\widehat{S})$ is satisfied by an induction on $\dim(\widehat{S})$. By Lemma \ref{reducetomixedchar} and Lemma \ref{reducetolocal}, we may assume that $\widehat{S}$ is a local integral scheme that is essentially of finite type over $B$ with a characteristic $p$ residue field at the closed point. We give an argument below in the case that $\widehat{S}$ is flat over $B$. The only way this flatness fails to occur is if $\widehat{S}$ is an $\F_p$-algebra. The reader can check that all our arguments go through in this case as well, provided a few trivial modifications are made to Lemma \ref{geompresentlemma}. We prefer to not make these modifications here for clarity of exposition.

If $\dim(\widehat{S}) = 1$, then any alteration of $\widehat{S}$ can be dominated by a finite morphism, so there is nothing to show as finite morphisms have no higher cohomology. We may therefore assume that the relative dimension of $\widehat{S}$ over $B$ is at least $1$. 

With the assumptions as above, given an alteration $\widehat{f}:\widehat{X} \to \widehat{S}$, we want to find an alteration $\widehat{g}:\widehat{Y} \to \widehat{X}$ such that $\widehat{g}^*(H^i(\widehat{X},\calO)) \subset p(H^i(\widehat{Y},\calO))$. As $\widehat{f}$ is an alteration, the center $\widehat{Z} \subset \widehat{S}$ is a closed subset of codimension $\geq 2$ such that $\widehat{f}$ is finite away from $\widehat{Z}$. Applying the conclusion of Proposition \ref{geompresentlemma}, we can find a diagram
\[ \xymatrix{ s \in Z \ar[r]^{i_Z} \ar[rrd] & S \ar[rd] \ar[r]^j & \overline{S} \ar[d]^{\pi} & \partial\overline{S} \ar[ld] \ar[l]_i \\
			& & W & } \]
satisfying the conditions guaranteed by Proposition \ref{geompresentlemma}. The next step is to extend the alteration $\widehat{f}$ to an alteration $\overline{f}:\overline{X}\to \overline{S}$ which gives $\widehat{f}$ as the germ at $s$ and has center $Z \subset \overline{S}$. This can be accomplished as follows: normalising $\overline{S}-Z$ in the function field of $\widehat{X}$ gives rise to a finite morphism $X' \to \overline{S} - Z$ which agrees with $\widehat{f}$ over $\widehat{S} - \widehat{Z}$. Glueing $X'$ and $\widehat{X}$ along their fibres over $\widehat{S} - \widehat{Z}$ (and spreading out a little) defines an alteration $\overline{f}'$ of an open subset $U \subset \overline{S}$ satisfying the following:
\begin{itemize}
\item $\overline{S} - U \subset Z - (Z \cap \widehat{S})$, and therefore, $s \in U$.
\item $\overline{f}'$ agrees with $\widehat{f}$ at $s$, and $\overline{f}'$ is finite on $U - (Z \cap U)$.
\end{itemize}
By Nagata compactification (see \cite[Theorem 4.1]{ConradNagata}), we obtain an alteration $\overline{f}: \overline{X} \to \overline{S}$ which is finite away from $Z$ and agrees with $\widehat{f}$ over $\widehat{S}$. Let $f:X \to S$ denote the restriction of $\overline{f}$ to $S$. We summarise the preceding constructions by the following diagram:
\[ \xymatrix{ X \times_S Z \ar[r] \ar[d]^{f_Z} & X \ar[d]^f \ar[r]^{j_X} & \overline{X} \ar[d]^{\overline{f}} & \Delta = \partial\overline{S} \times_{\overline{S}} \overline{X} \ar[l]_-{i_X} \ar[d]^{f_{\partial\overline{S}}} \\ 
			s \in Z \ar[r]^{i_Z} \ar[rrd] & S \ar[rd] \ar[r]^j & \overline{S} \ar[d]^{\pi} & \partial\overline{S} \ar[ld] \ar[l]_i \\
			& & W & } \]
Here the first row is obtained by base change from the second row via  $\overline{f}$. In particular, $i_X$ is the inclusion of a Cartier divisor. As $\overline{f}$ is finite away from the closed set $Z$ which does not meet $\partial\overline{S}$, the map $f_{\partial\overline{S}}$ is finite. In particular, the scheme $\partial\overline{S} \times_{\overline{S}} \overline{X}$ is affine. Applying Proposition \ref{cohvanprop} to the map $i_X$, we find that $H^i(\overline{X},\calO) \to H^i(X,\calO)$ is surjective for $i > 0$. Since $\dim(W) < \dim(S)$, the inductive hypothesis and Proposition \ref{reducereldim} ensure that Condition $\calC_d(W)$ is true for all $d$. As $\overline{X} \to W$ is proper surjective, we can find an alteration $\overline{g}:\overline{Y} \to \overline{X}$ such that $\overline{g}^*(H^i(\overline{X},\calO)) \subset p(H^i(\overline{Y},\calO))$. It follows that a similar $p$-divisibility statement holds for the alteration $g:Y \to X$ obtained by restricting $\overline{g}$ to $X \hookrightarrow \overline{X}$. Lastly, by flat base change, we know that $H^i(X,\calO)$ generates $H^i(\widehat{X},\calO)$ as a module over $\Gamma(\widehat{S},\calO)$. Thus, pulling back this alteration to $\widehat{X} \hookrightarrow X$ produces the desired alteration $\widehat{g}:\widehat{Y} \to \widehat{X}$.
\end{proof}

\begin{remark}
One noteworthy feature of the proof of Proposition \ref{killbyalt} is the following: while trying to show $\calC_0(S)$ is satisfied, we use that $\calC_d(S')$ is satisfied for $d > 0$ and certain affine schemes $S'$ with $\dim(S') < \dim(S)$. We are allowed to make such arguments thanks to Proposition \ref{reducereldim} and induction. However, this phenomenon explains why Proposition \ref{reducereldim} appears before Proposition \ref{killbyalt} in this paper, despite the relevant statements naturally preferring the opposite order.
\end{remark}

\begin{remark}
Theorem \ref{mixedcharpdiv}, while ostensibly being a statement about coherent cohomology, is actually motivic in that it admits obvious analogues for most natural cohomology theories such as de Rham cohomology or \'etale cohomology. For the former, one can use Theorem \ref{mixedcharpdiv} and the Hodge-to-de Rham spectral sequence to reduce to proving a $p$-divisibility statement for $H^i(X,\Omega^j_{X/S})$ with $j > 0$. Choosing local representatives for differential forms and extracting $p$-th roots out of the relevant functions can then be shown to solve the problem. In \'etale cohomology, there is an even stronger statement: for any noetherian excellent scheme $X$, there exist {\em finite} covers $\pi:Y \to X$ such that $\pi^*(H^i_\et(X,\Z_p)) \subset p(H^i_\et(Y,\Z_p))$ for any fixed $i > 0$; this statement follows from \cite[Theorem 1.1]{Bhattanngrpsch}  using the exact sequences of (continuous $p$-adic) \'etale sheaves
\[ 0 \to \Z_p \stackrel{p}{\to} \Z_p \to \Z/p \to 0.\]
We hope to find finite covers that work for coherent cohomology (see Remark \ref{rmk:findfinitecov}), but cannot do so yet.
\end{remark}

\begin{remark}
The proof of Theorem \ref{mixedcharpdiv} actually shows: given a proper morphism $f:X \to S$ with $S$ excellent, there exists wan alteration $\pi:Y \to S$ such that, with $g = \pi \circ f$, we have:
\begin{enumerate}
\item $\pi^*(\R^1 f_*\calO_X) \subset p (\R^1 g_*\calO_Y)$.
\item The map $\tau_{\geq 2} \R f_*\calO_X \to \tau_{\geq 2} \R g_*\calO_Y$ is divisible by $p$ as a morphism in $\D(\Coh(S))$.
\end{enumerate}
The reason one has to truncate above $2$ and not $1$ in the second statement above is that divisibility by $p$ in a $\Hom$-group imposes torsion conditions not visible when requiring individual classes to be divisible by $p$.  For instance, the second conclusion above implies that the $p$-torsion in $\R^i f_*\calO_X$ for $i \geq 2$ can be killed by alterations. We do not know how to prove an $H^1$-analogue of this statement: when $S$ is affine, this analogue amounts to verifying that functions $H^0(X,\calO_X/p)$ on the special fibre of $X$ lift to the functions $H^0(X,\calO_X)$ on all of $X$  provided we allow passage to alterations.
\end{remark}

We record a global corollary of Theorem \ref{mixedcharpdiv} that was already proven above.

\begin{corollary}
\label{mixedcharpdivglobal}
Let $f:X \to S$ be a proper morphism with $S$ excellent. Then there exists an alteration $\pi:Y \to X$ such that, with $g = f \circ \pi$, we have $\pi^*(\R^i f_*\calO_X) \subset p(\R^i g_*\calO_Y)$ for each $i > 0$.
\end{corollary}
\begin{proof}
One can trace through our constructions to see we have already proven this. Alternately, this follows by combining Theorem \ref{mixedcharpdiv} and Lemma \ref{reducetolocal}.
\end{proof}

Finally, we give an example showing that Theorem \ref{mixedcharpdiv} fails as soon as the properness of $f$ is relaxed.

\begin{example}
\label{ex:properneeded}
Let $k$ be a characteristic $p$ field, and let $X = \P^n_k - \{x\}$ for some $x \in \P^n(k)$ and $n \geq 2$.  Then $H^{n-1}(X,\calO_X) \simeq H^n_{x}(\P^n_k,\calO_{\P^n_k})$ is non-zero. Moreover, for any proper surjective morphism $\pi:Y \to X$, the pullback $\calO_X \to \R \pi_* \calO_Y$ is a direct summand (by \cite{ddscposchar}), so $H^{n-1}(X,\calO_X) \to H^{n-1}(Y,\calO_Y)$ is also a direct summand. In particular, non-zero classes in $H^{n-1}(X,\calO_X)$ cannot be killed by proper covers. Replacing $X$ with the obvious mixed characteristic variant $X'$ gives an example of a $\Z_p$-flat scheme $X'$ with non-zero higher coherent cohomology that cannot be made divisible by $p$ on passage to proper covers.
\end{example}

\section{A stronger result in positive characteristic}
\label{sec:mixedimpliespos}

Our goal in this section is to explain an alternative proof of \cite[Theorem 1.5]{ddscposchar} using Theorem \ref{mixedcharpdiv}. Recall that the former asserts:

\begin{theorem}
\label{thm:poscharkillfinite}	
Let $f:X \to S$ be a proper morphism of noetherian $\F_p$-schemes. Then there exists a finite surjective map $\pi:Y \to X$ such that, with $g = f \circ \pi$, the pullback $\pi^*: \tau_{\geq 1} \R f_* \calO_X \to \tau_{\geq 1} \R g_* \calO_Y$ is $0$.
\end{theorem}

Applying Theorem \ref{mixedcharpdiv} in positive characteristic, {\em a priori}, only allows us to kill cohomology on passage to proper covers. The point of the proof below, therefore, is that annihilation by proper covers implies annihilation by finite covers for coherent cohomology; see  \cite[\S 6]{Bhattanngrpsch} for an example with \'etale cohomology with coefficients in an abelian variety where such an implication fails.

\begin{proof}[Proof of Theorem \ref{thm:poscharkillfinite}]
We first explain the idea informally. Using Corollary \ref{mixedcharpdivglobal}, one finds proper surjective maps $Y' \to X$ and $Y'' \to Y'$ annihilating the higher coherent cohomology of $X \to S$ and $Y' \to X$ respectively; then one simply checks that the Stein factorisation of $Y'' \to X$ does the job.

In more detail, by repeatedly applying Corollary \ref{mixedcharpdivglobal} and using elementary facts about derived categories (see \cite[Lemma 3.2]{ddscposchar}), we may find a proper surjective map $\pi':Y' \to X$ such that, with $g' = f \circ \pi'$, the pullback $\tau_{\geq 1} \R f_*\calO_X \to \tau_{\geq 1} \R g'_*\calO_Y'$ is $0$.  Applying the same reasoning now to the map $\pi':Y' \to X$, we find a map $\pi^\flat:Y'' \to Y'$ such that, with $\pi'' = \pi^\flat \circ \pi'$, we have that $\tau_{\geq 1} \R \pi'_* \calO_Y \to \tau_{\geq 1} \R \pi''_*\calO_{Y''}$ is $0$. The picture obtained thus far is:
\[ \xymatrix{ Y'' \ar[r]^{\pi^\flat} \ar[rd]^{\pi''} & Y' \ar[d]^{\pi'} \ar[rd]^{g'} & \\
			& X \ar[r]^f & S. } \]
The diagram restricted to $X$ gives rise to the following commutative diagram of exact triangles in $\D(\Coh(X))$:
\[ \xymatrix{ \calO_X \ar@{=}[r] \ar[d] & \calO_X \ar[d] \ar[r] & 0 \ar[r] \ar[d] & \calO_X[1] \ar[d] \\
\pi'_*\calO_{Y'} \ar[r] \ar[d]^a & \R \pi'_*\calO_{Y'} \ar[r] \ar[d]^b \ar@{.>}[ld]_s & \tau_{\geq 1} \R \pi'_* \calO_{Y'} \ar[r] \ar[d]^{c = 0} & \pi'_*\calO_{Y'}[1] \ar[d]^{a[1]} \\
\pi''_*\calO_{Y''} \ar[r]  & \R \pi''_*\calO_{Y''} \ar[r]  & \tau_{\geq 1} \R \pi''_* \calO_{Y''} \ar[r] & \pi''_*\calO_{Y''}[1]. } \]
Here the vertical arrows are the natural pullback maps, and the dotted arrow $s$ is a chosen lifting of $b$ guaranteed by the condition $c = 0$ (which is true by construction). Applying $\R f_*$ to the above diagram, we find a factorisation:
\[ \xymatrix{ \R f_*\calO_X \ar[rr]^h \ar[rd]^d & & \R f_* (\pi''_*\calO_{Y''}) \\
					& \R (f \circ \pi')_*\calO_{Y'}  \simeq \R g'_*\calO_{Y'} \ar[ru]^e. & } \]
The map $d$ induces the $0$ map on $\tau_{\geq 1}$ by construction. It follows that the same is true for the map $h$. On the other hand, the sheaf $\pi''_* \calO_{Y''}$ is a coherent sheaf of algebras on $X$. Hence, it corresponds to a finite morphism $\pi:Y \to X$. In fact, $\pi$ is simply the Stein factorisation of $\pi''$. In particular, $\pi$ is surjective. It then follows that $\pi:Y \to X$ is a finite surjective morphism such that, with $g = f \circ \pi$, the induced map $\tau_{\geq 1} \R f_*\calO_X \to \tau_{\geq 1} \R g_*\calO_Y$ is $0$, as desired.
\end{proof}

\begin{remark}
There is an alternative and more conceptual explanation of the preceding reduction from proper covers to finite covers in the case of $H^1$. Namely, let $\alpha \in H^1(X,\calO_X)$ be a cohomology class, and let $f:Y \to X$ be a proper surjective map such that $f^*\alpha = 0$. We may represent $\alpha$ as a $\G_a$-torsor $T \to X$. The assumption on $Y$ then says that there is an $X$-map $Y \to T$. By the defining property of the Stein factorisation $Y \to Y'' \to X$, the map $Y \to T$ factors as a map $Y'' \to T$, i.e., the pullback of $T$ (or, equivalently $\alpha$) along the finite surjective map $Y'' \to X$ is the trivial torsor, as wanted.  The key cohomological idea underlying this argument is that the pullback $H^1(Y'',\calO_{Y''}) \to H^1(Y,\calO_Y)$ is injective. This injectivity fails for higher cohomological degree, so the proof above is slightly more complicated.
\end{remark}

\begin{remark}
\label{rmk:findfinitecov}
Assume for a moment that the conclusion of Theorem \ref{mixedcharpdiv} can be lifted to the derived category as discussed in Remark \ref{rmk:ptorsionkill}, i.e., we can kill $p$-torsion in higher coherent cohomology by passage to alterations. Then the argument given in the proof of Theorem \ref{thm:poscharkillfinite} applies directly to show that, in fact, one can make cohomology $p$-divisible (in the derived sense) by passage to finite covers. In particular, we can then replace ``alteration'' with ``finite surjective map'' in the statement of Theorem \ref{mixedcharpdiv}. We have checked this consequence in a few non-trivial examples (like the blowup of an elliptic $2$-dimensional singularity over $\Z_p$), and we hope that it is a reasonable expectation in general.
\end{remark}

\bibliography{mixedchar}

\end{document}